\documentclass[reqno,english]{amsart}
\usepackage{amsfonts,amsmath,latexsym,verbatim,amscd,mathrsfs,color,array}
\usepackage[colorlinks=true]{hyperref}

\usepackage{amsmath,amssymb,amsthm,amsfonts,graphicx,color}
\usepackage{amssymb}
\usepackage{pdfsync}
\usepackage{epstopdf}
\usepackage[colorlinks=true]{hyperref}

\makeatletter
\def\@currentlabel{2.1}\label{e:dispaa}
\def\@currentlabel{2.21}\label{e:dispau}
\def\@currentlabel{2.22}\label{e:dispav}
\def\@currentlabel{2.23}\label{e:dispaw}
\def\@currentlabel{2.24}\label{e:dispax}
\def\theequation{\thesection.\@arabic\c@equation}
\makeatother
 
\numberwithin{equation}{section}

\newcommand{\rem}{{Rm}}

\newcommand{\bla}{{\lambda,\lambda',h,h',k}}

\newcommand{\R} {\mathbb R}
\newcommand{\cuad}{{\sqcap\kern-.68em\sqcup}}

\newcommand{\uk}{u_{{_K}}}
\newcommand{\p}{\frac{n+2}{n-2}}

\newcommand{\mmm}{\frac{4}{n-2}}

\newcommand{\vp}{\varphi}
\newcommand{\vP}{\varPhi}
\newcommand{\vps}{\psi}

\newcommand{\be}{\begin{equation}}
\newcommand{\bee}{\begin{equation*}}
\newcommand{\ee}{\end{equation}}
\newcommand{\eee}{\end{equation*}}
\newcommand{\bs}{\begin{split}}
\newcommand{\esp}{\end{split}}

\newcommand{\la}{\lambda}

\renewcommand{\theequation}{\thesection.\arabic{equation}}

\newtheorem{cor}{Corollary}[section]
\newtheorem{defn}{Definition}[section]
\newtheorem{thm}{Theorem}[section]

\newtheorem{remark}{Remark}[section]

\newtheorem{lemma}{Lemma}[section]
\newtheorem{case}{Case}

\newcommand{\bremark}{\begin{remark} \em}
\newcommand{\eremark}{\end{remark} }

\title{ New type I ancient compact solutions of  the \\ Yamabe flow}

\author{Panagiota Daskalopoulos}
\address{ {\bf P. Daskalopoulos:} Department of Mathematics, Columbia University, 2990 Broadway, New York, NY 10027, USA.}
\email{pdaskalo@math.columbia.edu}

\author{Manuel del Pino}
\address{ {\bf M. del Pino:} Departamento de Ingenier\'{\i}a Matem\'atica and CMM, Universidad de Chile, Casilla 170 Correo 3, Santiago, Chile. } \email{delpino@dim.uchile.cl
}

\author{John King}
\address{ {\bf J. King:}  School of Mathematical Sciences, The University of Nottingham, University Park,
Nottingham, NG7 2RD}\email{john.king@nottingham.ac.uk
}

\author{Natasa Sesum}
\address{ {\bf N. Sesum:} Department of Mathematics, Rutgers University, 110 Frelinghuysen road, Piscataway,  NJ 08854, USA.}\email{natasas@math.rutgers.edu
}

\begin{document}

\maketitle

\begin{abstract}
We construct new  ancient compact solutions to the  Yamabe flow.  Our solutions
are rotationally symmetric and converge, as $t \to -\infty$, to  two self-similar complete non-compact  solutions to the Yamabe flow 
moving  in opposite directions. They are type I ancient solutions.

\end{abstract}

\section{Introduction}\label{sec-into}

Let $(M,g_0)$ be a compact manifold without boundary  of dimension $n \geq 3$. If $g = \bar \vp^{\frac 4{n-2}} \, g_0$
is a metric conformal to $g_0$, the scalar curvature $R$  of $g$ is given in terms of the
scalar curvature $R_0$ of $g_0$ by
$$R= \bar \vp^{-\frac {n+2}{n-2}} \, \big ( - \bar c_n \Delta_{g_0} \bar \varphi  + R_0 \, \bar \vp \big )$$
where $\Delta_{g_0}$ denotes the Laplace Beltrami operator with respect to $g_0$ and $\bar c_N =  4 (n-1)/(n-2)$.

In 1989 R. Hamilton introduced the  {\em Yamabe  flow}
\begin{equation}
\label{eq-YF}
\frac{\partial g}{\partial t} = -R\, g
\end{equation}
as an approach to give yet another proof of the {\em Yamabe problem}  on manifolds of positive conformal Yamabe invariant, using the flow and the parabolic techniques.
The Yamabe flow \eqref{eq-YF} is the negative $L^2$-gradient flow of the total scalar curvature, restricted to a given conformal class. This was shown by S. Brendle \cite{S2, S1} (up to a technical condition in dim $n \geq 6$).  Significant earlier works  in this directions include those by R. Hamilton \cite{H}, B. Chow \cite{Ch}, R. Ye \cite{Y},
H. Schwetlick and M. Struwe \cite{SS} among many others. The Yamabe conjecture, was previously  shown by R. Shoen 
via elliptic methods in his seminal work \cite{S}.

\smallskip 
In the special case where the background manifold $M_0$ is the  sphere $S^n$ and $g_0$ is the standard
spherical metric $g_{_{S^n}}$, the Yamabe flow evolving a metric $g= \bar \vp^{\frac 4{n-2}}(\cdot,t)
 \, g_{_{S^n}}$ takes (after  rescaling in time by a constant)
	the form of the {\em  fast diffusion equation}
\begin{equation}
\label{eq-YFS}
(\bar \vp^\frac{n+2}{n-2})_t  = \Delta_{S^n} \bar \vp -c_n \bar \vp,  \qquad c_n = \frac{n(n-2)}{4}.
\end{equation}
Starting with any smooth metric $g_0$ on $S^n$, it follows by the results in \cite{Ch}, \cite{Y} and \cite{DS} that
the solution of \eqref{eq-YFS} with initial data $g_0$ will become singular at some finite time $t < T$ and
$g$ becomes spherical at time $T$,  which means that after
a normalization, the normalized flow converges to the spherical metric.  In addition, $\bar \vp$ becomes extinct  at $T$.

A metric $g =  \bar \vp^{\frac 4{n-2}} \, g_{_{S^n}}$ may also be expressed as a metric on $\R^n$ via  stereographic
projection. It follows that if $g =  \hat \vp^{\frac 4{n-2}} (\cdot,t) \, g_{_{\R^n}}$ (where $g_{_{\R^n}}$ denotes the standard metric
on $\R^n$)  evolves by the Yamabe flow \eqref{eq-YF},  then $\hat  \vp$ satisfies (up to rescaling in time by a constant) the fast diffusion equation on $\R^n$
\begin{equation}
\label{eq-ufd}
(\hat \vp^p)_t = \Delta \hat \vp, \quad \qquad p:= \frac{n+2}{n-2}.
\end{equation}
Observe that if $g =  \hat \vp^{\frac 4{n-2}} (\cdot,t) \, g_{_{\R^n}}$ represents a smooth solution when lifted to $S^n$,
then $\hat{\vp}(\cdot,t)$ satisfies the growth condition
$$\hat \vp(z,t) = O  (|z|^{-(n-2)}), \qquad \mbox{as} \,\, |z| \to \infty.$$

\medskip

\begin{defn}[Type I and type II ancient solutions]\label{defn-ancient}
The solution $g = \bar \vp(\cdot, t)^{\frac 4{n-2}}\, g_0$ to \eqref{eq-YF}  is called ancient if it exists for all time $t\in (-\infty, T)$,
where $T < \infty$.   We will say that the ancient solution $g$ is
type I,  if its  Riemannian curvature  satisfies
$$\limsup_{t\to-\infty} \, ( |t| \, \max_{M_0} |\mbox{\em Rm}| \ \, (\cdot,t))< \infty.$$
An ancient  solution which is not type I,  will be called type
II.

\end{defn}

\smallskip

\noindent The simplest example  of  an ancient solution  to the Yamabe  flow on $S^n$ is 
the  family of {\em  contracting spheres}. 
They are  special solutions  $\bar \vp$   of \eqref{eq-YFS} which depend only on time $t$ and satisfy the ODE
$$\frac {d \bar \vp^{\p}}{dt} =-c_n\, \bar \vp.$$ They are given by
 \begin{equation}
\label{eq-spheres}
\bar \vp_{_{S}}(p,t) = \left ( \mmm \, c_n\,  (T-t) \right )^{\frac {n-2}4}, \qquad p \in S^n
\end{equation}
and  represent 
a sequence of round spheres shrinking to a point at time $t=T$.
They are shrinking solitons and  type I ancient solutions.

\smallskip

{\em  King  solutions:} They were  discovered
 by J.R. King \cite{K1}.  They can be expressed on $\R^n$ in closed from (after stereographic projection), namely 
 $g= \hat  \vp_{_{K}} (\cdot,t)^{\frac 4{n-2}} \, g_{_{\R^n}}$, where $\hat \vp_{_{K}}$ is the radial function
 \begin{equation}
\label{eq-king}
\hat \vp_{_{K}}(z,t) = \left(\frac{  a(t) }{1 + 2  b(t) \, |z|^2 + |z|^4}\right)^{\frac{n-2}{4}}, \qquad z \in \R^n 
\end{equation}
and the coefficients  $a(t)$ and $ b(t)$ satisfy  a certain system of ODEs.
The King solutions are {\em not
solitons} and  may be  visualized, as $t \to -\infty$,    as two Barenblatt self-similar solutions  ''glued'' together
to form a compact solution to the Yamabe  flow. They are type I ancient
solutions.

\medskip
Let us make the analogy  with the Ricci flow on $S^2$. The two  explicit compact ancient solutions to the two dimensional Ricci flow are the contracting spheres   and the King-Rosenau solution
\cite{K1}, \cite{K2}, \cite{R}. The latter one  is  the analogue of the King solution (\ref{eq-king}) of  the Yamabe flow. The difference is that the King-Rosenau Ricci flow ancient solution   is type II,   while
the King Yamabe flow solution  is type I.

It has been showed by Daskalopoulos, Hamilton and Sesum  \cite{DS1} that the spheres and the King-Rosenau solution  are the only compact ancient solutions  to the two dimensional Ricci flow. The natural question to raise is whether the analogous statement holds  true for the Yamabe flow, that is, whether the contracting  spheres and the King solution are the only compact ancient solutions to the Yamabe flow. This occurs not to be the case as the following discussion shows.

\medskip
Indeed, in \cite{DDS} the existence of a new class of type II   ancient radially symmetric solutions of the Yamabe flow \eqref{eq-YFS}
on $S^n$  was shown. These  new solutions, as $ t \to -\infty$, may be visualized as two spheres joined by a short neck. 
Their curvature operator changes sign. We will refer to them as {\em towers of moving bubbles}.

\medskip

Since the towers of moving bubbles are shown to be type II ancient solutions, while the contracting spheres and  the King solutions are of  type I, 
one may still ask whether the latter two  are
the only ancient compact type I solutions of  the Yamabe flow on $S^n$, equation \eqref{eq-YFS}. In this work we will observe that this is not the case,  
as will show  the existence
of other ancient compact type I solutions on $S^n$. 

\medskip

It is simpler to work  in cylindrical coordinates, so let us first describe the coordinate change. 
Let  $g=\hat \vp^{\frac 4{n-2}} (\cdot,t) \, g_{_{\R^n}}$ be  a radially symmetric  solution of \eqref{eq-ufd}.
For any $T >0$    the  cylindrical change of variables is given by 
\be
\label{eqn-cv1}
\vp(x,\tau )   = (T-t)^{-\frac 1{p-1}} r^{\frac 2{p-1}} \, \hat \vp(y,t)  , \quad x=\ln  |y|,\,  \tau  = - \ln (T-t).
\ee
In this language equation \eqref{eq-ufd}  becomes
\begin{equation}\label{eqn-vv}
(\vp^p)_\tau =  \vp_{xx} +  \alpha^{-1}  \vp^p -  \beta \vp, \quad \beta = \frac {(n-2)^2}4 ,\quad \alpha = \frac {p-1}p = \frac {n+2}4.
\end{equation}
By  suitable scaling we can make the two constants $\alpha$ and $\beta$ in \eqref{eqn-vv} equal to 1, so that from now on
we will consider  the equation
\begin{equation}\label{eqn-vp}
(\vp^p)_\tau =  \vp_{xx} +  \vp^p -   \vp.
\end{equation}
 
\medskip

It is well known (c.f. in   \cite{V}, Section 3.2.2 and \cite{DS2, DKS})  that for any given $\lambda \geq  0$  
 equation \eqref{eqn-vp} admits
an one parameter family  of traveling wave  solutions of the form 
$$\phi_\la(x,\tau)=\psi_\la(x-\lambda \, \tau), \qquad y:=x-\la\, \tau$$
which admit  the behavior    
\be
\label{eqn-behavior1}
\vps_\la(y)  = O(e^{y}), \quad \mbox{as}\,\,\  y \to - \infty.
\ee 
It follows that
$\vps:=\vps_\la$ satisfies the equation
\begin{equation}\label{eqn-vp2}
\vps_{yy} + \lambda \, (\vps^p)_y +  \vps^p  -  \vps =0
\end{equation}
and they are unique  up to translations of the self-similar variable $y$, given the \eqref{eqn-behavior1}.  The solutions 
$\vp_\la$ define  Yamabe shrinking solitons which correspond to smooth self-similar solutions of 
\eqref{eq-ufd}  when expressed  as  metrics on $\R^n$ (the smoothness follows from condition \eqref{eqn-behavior1}). 
It was shown in \cite{DS2} that they are type I ancient solutions. 
Solutions of \eqref{eqn-vp2} with  $\lambda =0$ correspond to  the steady states of equation \eqref{eqn-vp} and 
they  represent geometrically the standard metric  on the  sphere. 

When $\lambda >0$,  solutions to \eqref{eqn-vp2}  with behavior 
 \eqref{eqn-behavior1} define  smooth complete and non-compact Yamabe solitons
(shrinkers)  which all have {\em cylindrical behavior at infinity}, namely
$$\vps_\lambda(y) = 1 + o(1), \qquad \mbox{as}\,\, y \to + \infty.$$

In \cite{DKS} the asymptotic behavior, up to second order,   of these solutions was shown. 
Let  us next describe this behavior for the case $\la \geq 1$ which  will be heavily used in this work. 
For values of $\lambda$ in the range $0 < \lambda < 1$, the behavior
of the solutions $v_\lambda$ was also studied in \cite{DKS} and it is more complex as it differs for dimensions $3 \leq  N  \leq 6$ and $N  \geq 6$. 

When $\la \geq 1$, it is shown in  Theorem 1.1 in   \cite{DKS},   that    there exists 
a unique solution $\vps_{\la}$  of  \eqref{eqn-vp2} which is monotone increasing, satisfies 
\be
\label{eqn-vp0}
\vps_\la(0)= \frac 12
\ee
and has the asymptotic behavior 
\be
\label{eqn-vpla}
\vps_{\la}(y) = O(e^{y}), \quad \mbox{as}\,\,\  y \to - \infty \quad \mbox{and} \quad  \vps_{\la}(y) = 1 - C_\la\, e^{-\gamma_\la y} +
o(e^{-\gamma_\la y}), \quad \mbox{as}\,\,\  y \to +\infty
\ee
for some constants $\gamma_\la >0$ and $C_\la >0$ (depending on $\la$).  The exponent  $\gamma_\la$ satisfies the equation 
\be
\label{eqn-eqgamma}
\gamma^2 - \la p \gamma + (p-1) = 0
\ee
and for $\la >1$  is given by the smallest of the roots of this equation, that is
\begin{equation}
\label{eq-gamma-la}
\gamma_\la = \frac{\la p - \sqrt{\la^2 p^2 - 4 (p-1)}}{2} >0.
\end{equation}
When $\la=1$, equation \eqref{eqn-vp2} admits the  explicit {\em Barenblatt } solution 
\be\label{eqn-baren1}
\vps_1(y) = \left ( \frac 1{1+ c_p \, e^{-(p-1)\,y} }\right )^{1/(p-1)}, \qquad p-1=\frac 1{n-2}
\ee
with  $c_p$ chosen so that  $\vps_1(0)=1/2$. 

\medskip

We will establish  in this work the {\em existence of a five parameter family}  of solutions $\vp_\bla$ of \eqref{eqn-vp} with
$(\bla) \in \R^5$ and $\la, \la' >1$, $k \geq 0$. Let us next summarize our construction. 
First, note that because of the nonlinearity of  the time derivative in  \eqref{eqn-vp} it 
is more natural to define our  solutions  in terms of  the {\em pressure}  function
$$u:= \vp^{-\frac 4{n-2}}= \vp^{-(p-1)}$$
which satisfies the equation
\be\label{eqn-u}
p\, u_\tau = u\, u_{xx} - \frac p{p-1} \,  u_x^2 + (p-1) \, \big ( u^2 - u \big ). 
\ee
This becomes also apparent when one looks at  the King solutions \eqref{eq-king}, 
which in terms
of the pressure function $u:=\vp_{_{K}}^{-(p-1)}$ become polynomials in the radial  variable $r=|z|$. 

In terms of the pressure function,  equation  \eqref{eqn-u} 
admits an one parameter family (unique up to translations) of traveling wave  solutions  of the form $u_\lambda(x,\tau) = v_\la(x-\lambda \, \tau)$  with profile $v_\la = \vps_\la^{-(p-1)}$, with $\psi_\lambda$ the solution of \eqref{eqn-vp2} as described above. It follows that
$v:=v_\la$ satisfies the equation
\begin{equation}\label{eqn-v2}
v\, v_{yy}   - \frac p{p-1}  \, v_y^2  +  \lambda \, p\, v_y+  (p-1) \, (v^2 -  v) =0.
\end{equation}
Imposing the condition that 
$$v_\la(0)=2$$
which is equivalent to $\vps_\lambda(0)=1/2$, 
 it follows from  \eqref{eqn-vpla} that each $v_\la$ satisfies  the asymptotic behavior 
\be\label{eqn-vla}
v_{\la}(y) = O(e^{-(p-1)\, y}), \quad \mbox{as}\,\,\  y \to - \infty \quad \mbox{and} \quad  v_{\la}(y) = 1 + C_\la\, e^{-\gamma_\la y} +
o(e^{-\gamma_\la y}), \quad \mbox{as}\,\,\  y \to +\infty
\ee
with $\gamma_\la$ satisfying \eqref{eqn-eqgamma}-\eqref{eq-gamma-la} and $C_\la$ a constant  uniquely determined in terms of $\la$ and $p$.  
In addition each $v_\la$ is monotone decreasing in $y$, since it is known that $\vps_\la$ is monotone increasing. 
Notice also that in terms of the pressure function $\vp_1:=\vps_1^{-(p-1)}$ the Barenblatt  solution 
in  \eqref{eqn-baren1}  are given by
\be\label{eqn-baren2}
v_1(y)=1+ c_p \, e^{-(p-1)\,y}. 
\ee

\medskip

We will next describe the {\em building blocks}  in  our  construction of the five parameter family $u_\bla$ of ancient solutions of equation \eqref{eqn-u}, which will
be the main focus in this work.  Assume that 
\be
\label{eqn-ula0}
u_{\la,h}(x,\tau) := v_\la(x - \la \tau  + h)
\ee
is a traveling wave solution of \eqref{eqn-u} for a parameter $\la > 1$ and $h \in \R$
and let $u_{\la',h'}$ be another such solution  for a different choice of parameters $\la' >1$ and $h' \in \R$.  
Since equation \eqref{eqn-u} is invariant under reflection $x \to -x$, it follows that 
\be
\label{eqn-ula00}
\hat u_{\la',h'}(x,\tau) := u_{\la',h'}(-x,t) = v_\la(-x - \la' \tau + h')
\ee
is also a solution to \eqref{eqn-u}. It   corresponds to another  traveling wave of \eqref{eqn-v2} 
which travels in the opposite direction than $u_{\lambda',h'}$. 
The solution $u_{\la,h}(\cdot, \tau)$ is monotone decreasing in $x$ while $\hat u_{\la',h'}(\cdot, \tau)$ is monotone increasing.
Moreover, it  follows from \eqref{eqn-vla} that $u_{\la,h}, \hat u_{\la',h'}$ satisfy the asymptotics 
\be
\label{eqn-ula1}
u_{\la,h}(x,\tau) = O(e^{ - (p-1) \, x}), \,\,\,    \mbox{as}\,\,    x \to - \infty \qquad \mbox{and} 
\qquad   \hat u_{\la',h'}(x,\tau) = O(e^{ x \, (p-1) })
\, \, \,  \mbox{as}\, \,  x \to +\infty
\ee
and also 
\be
\label{eqn-ula2}
u_{\la,h}(x,\tau) = 1 + C_\la\, e^{-\gamma_\la (x-\la \tau+h)} +
o(e^{-\gamma_\la (x-\la \tau+h)}), \quad    \mbox{as}\, \,  x-\la \tau +h \to +\infty. 
\ee
and 
\be
\label{eqn-ula3}
 \hat u_{\la',h'}(x,\tau) = 1 + C_\la\, e^{\gamma_{\la'}  (x+\la' \tau-h)} + o(e^{\gamma_{\la'}(x+\la' \tau-h)}), \quad  \mbox{as}\,\,  x + \la' \tau -h  \to -\infty. 
\ee
with $\gamma_\la, \gamma_{\lambda'}$ given by \eqref{eq-gamma-la} and $C_{\lambda} >0, C_{\la'} >0$ depending  only on $\lambda$.  

Equation \eqref{eqn-u} also admits an one parameter family of space independent ancient solutions $\xi_k:=\xi_k(\tau)$ which are 
the solutions of the ode
\be\label{eqn-ode}
 \frac d{d\tau} \xi (\tau) =  \frac{p-1}{p} \, \big ( \xi^2(\tau) - \xi (\tau) \big )
\ee
and correspond to the {\em cylindrical } solution  of the Yamabe flow. 
Solving this equation gives 
\be\label{eqn-xik}
\xi_k(\tau) = \frac 1{ 1  - k  \, e^{\frac  {p-1}p\,  \tau }}. 
\ee
for a constant $k \in \R$. From now on  will take $k \geq 0$ which implies that  $\xi_k(\tau) \geq 1$ and $\xi_k(\tau) >1$, for $k >0$. 
Also, since $p >1$, we have
\be\label{eqn-xik2}
\xi_k(\tau) = 1 +  k  \, e^{\frac  {p-1}p\,  \tau } + O \big ( e^{\frac  {2 (p-1)}p\,  \tau } \big ), \qquad \mbox{as} \,\, \tau \to -\infty.
\ee
\medskip
We will show the existence of five    parameter class of ancient solutions $u_\bla$ of
equation \eqref{eqn-u} with $\la, \la'  >  1$, $k \geq 0$  and  $h, h' \in \R$, 
which as $\tau  \to -\infty$ may be visualized as the two traveling wave solutions,  $u_{\la,h}$ (traveling on the left) and $\hat u_{\la',h'}$ 
(traveling on the right) and a cylinder $\xi_k(\tau)$ in the middle.  In fact, we will show in  section \ref{sec-solutions} that for $\tau <<0$,  $u_\bla(\cdot,\tau)$ satisfies 
\be
\label{eqn-ubla0}
w^-_\bla \, (\cdot,\tau)  \leq u_\bla \, (\cdot, \tau)  \leq w_\bla^+ \, (\cdot, \tau)
\ee
with 
\be\label{eqn-vbla}
w_\bla^-:= \max \big ( v_\la(x-\la  \tau +h), \, \xi_k(\tau), \, v_{\la'}(- x - \la' \tau+h') \big )
\ee
and
\be\label{eqn-wbla}
w_\bla^+:= v_\la( x- \la  \tau \, (1 - q e^{ \frac {p-1}p \tau} )  +h) +  v_\la( - x- \la  \tau \, (1 - q  e^{ \frac {p-1}p\tau} )  +h) + \xi_k(\tau) - 2
\ee
with $q=q(p) >0$. 
It is clear that $w^-_\bla$ is a {\em subsolution}  of equation \eqref{eqn-u}. The main observation in our construction is that $w_\bla^+$ is a {\em supersolution }  of equation \eqref{eqn-u} with the appropriate choice of the parameter $q$. 

\medskip
Equivalently $\vp_\bla :=u_\bla^{- \frac 1{p-1}}$ defines a solution of \eqref{eqn-vp} and we recall  that $p-1=\frac 4{n-2}$. 
Let  $g_\bla(\tau):=  \vp_\bla^\mmm \, (\cdot,\tau) \, g_{cyl}$ denote the  metric on the 
cylinder $\R  \times S^{n-1}$ defined in terms of $\vp_\bla$,  where   $g_{cyl}:= dx^2 + g_{_{S^{n-1}}}$ denotes the standard  
cylindrical metric. We have seen  that \eqref{eqn-vp} is equivalent to $g_\bla$ 
satisfying the rescaled Yamabe flow $g_\tau= -(R - 1)g$. In addition we will show that  $\vp_\bla$  when lifted on $S^n$   defines a smooth ancient type I  solution to the Yamabe flow on $S^n \times (-\infty,T)$. 
Our main result is summarized as follows.

\begin{thm}\label{thm} For any $(\bla) \in \R^5$ such that $\la, \la' >1$, $k \geq 0$,  there exists an ancient solution $\vp_\bla$ of \eqref{eqn-vp} 
defined on $\R \times (-\infty,T)$,  for some $T=T_\bla \in (-\infty, +\infty]$, which as $\tau  \to -\infty$ may be visualized as the two traveling wave solutions,  $u_{\la,h}$ (traveling on the left) and $\hat u_{\la',h'}$ 
(traveling on the right) and a cylinder $\xi_k(\tau)$ in the middle. More precisely, it satisfies the bound 
$$
(w^-_\bla)^{-\frac {n-2}4} \, (\cdot, \tau) \leq \vp_\bla \, (\cdot, \tau)  \leq (w_\bla^+)^{-\frac {n-2}4} (\cdot, \tau), \qquad \mbox{for} \,\, \tau <<0. 
$$
The metric $g_\bla(\tau) := \vp^\mmm_\bla \, (\cdot, \tau) \, g_{cyl}$  when lifted on $S^n$   defines a smooth ancient  solution of the rescaled Yamabe flow $g_\tau = -(R-1)\, g$,
on $S^n \times  (-\infty, T)$. This is a  type I ancient solution in the sense that the norm of its curvature operator  is  uniformly bounded in time  $- \infty < \tau \leq \tau_0$, for all $\tau_0 < T$. 
\end{thm}

\smallskip

\bremark  It follows from Theorem \ref{thm} that  the corresponding solution to the unrescaled Yamabe flow \eqref{eq-YF} is a type I ancient solution in the sense of Definition \ref{defn-ancient}.

\eremark

\bremark[Similarities with the KPP equation and further discussion] Equation \eqref{eqn-vp}  resembles  the well known  semilinear KPP equation
\be\label{eqn-kpp} 
\vp_\tau = \vp_{xx} + f(\vp)
\ee
for a nonlinearity $f(\vp)$ which satisfies certain growth assumptions including $f(\vp)= \vp^p-\vp$ as in \eqref{eqn-vp}.  It is well known that equation \eqref{eqn-kpp}
possesses a family of traveling wave solutions $\psi_\la$, $\la \geq \la_*$ with similar behavior as those of equation \eqref{eqn-vp2} described above. 
F. Hamel and N. Nadirashvili   showed  in \cite{HN} the existence of  ancient solutions $\vp_\bla$ to \eqref{eqn-kpp} which are similar to
 those in   Theorem \ref{thm}. The method in \cite{HN}   exploits  the semilinear character of equation \eqref{eqn-kpp}. The error
of  approximation is estimated in a rather precise manner  by the corresponding  solution of  the linear equation 
$\zeta_\tau = \zeta_{xx} + f'(0)\, \zeta.$
This is done simply by using the heat kernel.  The same method cannot be applied to the  quasilinear equation  \eqref{eqn-vp}, which also  becomes {\em singular}  as
$x \to \pm \infty$. Despite the singular nature of our equation,  we are still able to obtain similar precise bounds  as in \cite{HN} by the construction of the appropriate  super-sub solutions. 

It would be interesting to  explore whether  similar  methods,  using barriers,  can be used  to provide   
the  construction  of  ancient solutions from solitons (self-similar solutions)  in  other parabolic PDE and in particular in  geometric flows.

\eremark

In our  previous  work  \cite{DDKS} we showed the existence of a four parameter family $g_{\la,\la',h,h'}$ of type I ancient solutions of the Yamabe flow
 \eqref{eq-YF}.   These 
solutions correspond  to the case $k=0$ in our five parameter family $g_\bla$ which is also covered by Theorem \ref{thm}. 
The methods of the construction in \cite{DDKS},  which are based on integral bounds and a'priori estimates, are
different than in this work. Our super-sub solution method of this work allows us to obtain much 
sharper  bounds on the solution as $\tau \to -\infty$. These bounds are  similar to those for the KPP equation in \cite{HN}.

\smallskip

The {\em outline}  of the paper is as follows. In section \ref{sec-king} we will review the construction  of the  King solutions  proven 
to exist in \cite{K1} and show their relevance 
to our construction. In   section \ref{sec-super} we will show that $w_\bla^+$ given by \eqref{eqn-wbla} defines a supersolution of equation \eqref{eqn-u} 
for an appropriate choice of parameters $q$ and $d$. This will constitute the main step in our proof. Section \ref{sec-solutions} will
be devoted to the proof of the existence of the ancient  solution $u_\bla$ as stated in Theorem \ref{thm}. In the final section \ref{sec-geom} we will show that our solution defines  a  type I ancient solution of the Yamabe flow.

\section{The King solutions}
\label{sec-king}

In this section we will review the existence and properties of  the  King solutions found in \cite{K1} and show the
similarities with our construction of $u_\bla$ as stated in Theorem \ref{thm}.  

One looks for a radially symmetric solution $\hat \vp_{_{K}}(z,t)$ of equation \eqref{eq-ufd} such that the pressure 
function $\hat u_{{_K}}:= \hat \vp_{_{K}}^{-(p-1)}$ is a polynomial in $r=|z|$ with coefficients depending on time $t$. It turns out that the polynomial 
is of degree four, that is,  $\hat u_{{_K}}$ has  the form
$$\hat u_{_{K}}(r,t):=  a(t)^{-1} \, ( 1 + 2  b (t)\,  r^2 + r^4), \qquad r=|z|, \,\, z \in \R^n$$
where  $ a(t),  b (t)$ are not defined explicitly but satisfy a system of ode's (c.f. in \cite{K1} and \cite{S3}). In terms of the rescaled 
variables and in cylindrical coordinates the King solutions are defined by a  pressure function of the form     
\be\label{eqn-king4}
\uk(x,\tau) = \xi(\tau)  +  \zeta(\tau)\, \cosh \, ((p-1) x), \qquad x \in \R
\ee
for some positive  functions $\zeta(\tau), \xi(\tau)$ of rescaled time $\tau$. The function $\uk$  satisfies  equation  \eqref{eqn-u} and a  direct calculation shows that
this is equivalent to the system of ode's 
\be\label{eqn-system}
\begin{cases} \, \, p\, \xi'(\tau) &=  p(p-1)\, \zeta^2(\tau) + (p-1)\, \big (\xi^2(\tau) - \xi(\tau) \big )\\
\,\,  p\, \zeta' (\tau) &=   (p-1)\,   \zeta(\tau)  \, \big ( -1   +   (p+1)  \, \xi(\tau) \big ) . 
\end{cases} 
\ee
In the case where $\zeta(\tau) \equiv 0$, the above system is equivalent to the ode \eqref{eqn-ode} 
which corresponds to space independent solutions of \eqref{eqn-u}. The non-trivial solution  
$\xi(\tau)$ of this equation  is given by \eqref{eqn-xik} for a parameter $k \geq 0$ and geometrically 
corresponds to cylinders evolving by the Yamabe flow. 

In all other cases, a simple analysis shows that 
\be\label{eqn-approx} 
\xi(\tau) = 1+ k\, e^{\frac {p-1}p \tau}  +  o( e^{\frac {p-1}p \tau} ), \qquad \zeta(\tau) = c\, e^{(p-1)\tau}  +
o(e^{(p-1)\tau}),
\ee
for some  parameters $k >0, c >0$. 

\medskip

Now let us relate the King solutions to the Barenblatt self-similar solutions \eqref{eqn-baren2} and our approximating supersolutions $w_\bla$ given by
\eqref{eqn-wbla}.   Consider the approximating solutions
defined by 
$$w_{K,h,k}(x,\tau) = v_1(x -\tau-h) + v_1(-x - \tau - h) + \xi_k(\tau) -2$$
with $v_1$ and $\xi_k$  given by \eqref{eqn-baren2} and \eqref{eqn-xik} respectively.  It follows that 
\be w_{K,h,k}(x,\tau) = c  \, e^{ (p-1)\, \tau} \, \cosh ((p-1) \, x) + \xi_k(\tau)
\ee
for a parameter $c$ depending on $h$. Using the asymptotics \eqref{eqn-xik2} it follows that 
$w_{K,h,k}(x,\tau)$ is the first order approximation of  \eqref{eqn-king4} as shown in \eqref{eqn-approx}.  
This shows that that  our approximations  $w_\bla$  given by \eqref{eqn-wbla}  is a natural  extension of the
King solutions to the case of $\lambda >1$ where the traveling waves (solitons) $v_\la$ are not given in
closed form. We will see in the next section that the exponentially small term $  \la q \, e^{ \frac {p-1}p \tau}$ 
which is subtracted from  the traveling wave speed $\lambda$ of the traveling wave $v_\lambda$ in the definition \eqref{eqn-wbla},
plays an essential role into making $w_\bla^+$ to be a supersolution.

\section{The supersolution $w_\bla^+$}
\label{sec-super}

In this section we will show that 
\be\label{eqn-wbla4}
w_\bla^+:= v_\la( x- \la  \tau \, (1 - q \, e^{ \frac {p-1}p \tau} )  +h) +  v_\la( - x- \la  \tau \, (1 - q  e^{ \frac {p-1}p\tau} )  +h) + \xi_k(\tau) - 2
\ee
is a supersolution of equation \eqref{eqn-u} for every $(\bla) \in \R^5$ with $\la, \la' >1$, $k \geq 0$ and  $h, h' \in \R$ if   $q=q(p) >0$ is chosen appropriately.  

\smallskip
To simplify the notation, we will fix for the remainder of this section $(\bla) \in \R^5$ as above and we  will simply denote by $w:=w_\bla^+$ and 
\[w_1(x,\tau) := v_{\la}(x - \la \tau (1 - q e^{\tau/2}) + h), \,\,\,  w_2 (\tau) := \xi_k(\tau)-1,  \,\,\, w_3(x,\tau) := v_{\la'} (-x - \la'\tau( 1 - q e^{\tau/2}) + h').\]
Also, we will set
\be\label{eqn-zbz}
z := x - \la \tau (1 - q e^{\frac{p-1}{p}\, \tau}) + h \quad \mbox{and} \quad \bar{z} := -x - \la'\tau (1 - q e^{\frac{p-1}{p}\, \tau}) + h'.
\ee
We recall that $v_\la, v_{\la'}$ satisfy \eqref{eqn-v2} and $\xi_k$ satisfies \eqref{eqn-ode}. It follows that $w_2:=\xi_k-1$ satisfies the ode
\begin{equation}
\label{eq-v2}
p\,(w_2)_{\tau} = (p-1) \big ( w_2^2 + w_2).
\end{equation} 
We then have by \eqref{eqn-wbla4} that 
\begin{equation}
\label{eq-def-w}
w :=  w_1 + w_2 + w_3 - 1.
\end{equation}
Define next the operator
\[L(w) := w\,  w_{xx} - \frac{p}{p-1}\,  w_x^2 + (p-1) \, \big (w^2 - w)  - p\, w_{\tau}.\]
A direct computation shows that 
\bee
\begin{split}
L(w) &= \left(\sum_{i=1}^3 w_i - 1\right)\, \sum_{i=1}^3 (w_i)_{xx} - \frac{p}{p-1}\left(\sum_{i=1}^3 (w_i)_x\right)^2 + (p-1)\, \left(\sum_{i=1}^3 w_i - 1\right)^2 \\
&- (p-1)\, \left(\sum_{i=1}^3 w_i - 1\right) - p\, \left(\sum_{i=1}^3 (w_i)_{\tau}\right) \\
&= \sum_{i,j=1}^3 w_i (w_j)_{xx} - \sum_{i=1}^3 (w_i)_{xx} - \frac{p}{p-1}\, \sum_{i,j=1}^3 (w_i)_x (w_j)_x + (p-1)\, \sum_{i,j=1}^3 w_i w_j  \\
&- 3\, (p-1)\sum_{i=1}^3 w_i + 2\, (p-1) - p\, \left(\sum_{i=1}^3 (w_i)_{\tau}\right).
\end{split}
\eee
Since 
$$\sum_{i=1}^3 (w_i)_{\tau} = \sum_{i=1,3} (w_i)_{\tau} + (w_2)_\tau = 
(w_2)_{\tau} -  (\la  v_\la'  + \la' \, v_{\la'} ) + 
q \, \la \, (\tau e^{\frac{p-1}{p}\, \tau})' \, v_\la'  +   q \, \la' \, (\tau e^{\frac{p-1}{p}\, \tau})' \, v_{\la'}'$$
and $(w_1)_x=v_\la'$, $(w_3)_x = -v_{\la'}'$  
we have 
$$\sum_{i=1}^3 (w_i)_{\tau} =  (w_2)_{\tau} -  (\la  (w_1)_x  -  \la' \, (w_3)_x) + 
q \, \la \, (\tau e^{\frac{p-1}{p}\, \tau})' \, (w_1)_x   -    q \, \la' \, (\tau e^{\frac{p-1}{p}\, \tau})' \, (w_3)_x.$$
Hence, 
\bee
\begin{split}
L(w) = &\sum_{i,j=1}^3 w_i (w_j)_{xx} - \sum_{i=1}^3 (w_i)_{xx} - \frac{p}{p-1}\, \sum_{i,j=1}^3 (w_i)_x (w_j)_x + (p-1)\, \sum_{i,j=1}^3 w_i w_j  \\
&- 3 \, (p-1)\sum_{i=1}^3 w_i + 2(p-1) - p\, (w_2)_{\tau} + p \, ( \la \, (w_1)_x - \la' \, (w_3)_x )\\
&-p q \, \la \, (\tau e^{\frac{p-1}{p}\, \tau})' \, (w_1)_x + p q \, \la'\,  (\tau e^{\frac{p-1}{p}\, \tau})' \, (w_3)_x. 
\end{split}
\eee
Since both  $v_\la, v_{\la'}$ satisfy \eqref{eqn-v2} and  $w_2$ satisfies \eqref{eq-v2}, the above yields
\bee
\begin{split}
L(w) &= w_1 (w_3)_{xx} + w_3 (w_1)_{xx} + (w_2 - 1)\, ((w_1)_{xx} + (w_3)_{xx}) - \frac{2p}{p-1} \, (w_1)_x (w_3)_x \\
&+ 2(p-1)\, (w_1 w_3 - 1) - 2 (p-1)\, \sum_{i=1,3} (w_i - 1) + 2 (p-1) \, w_2   (w_1 + w_3)  \\
 &-  4 (p-1) \, w_2 - p q \, \la \,  (\tau e^{\frac{p-1}{p}\, \tau})' \, (w_1)_x + p q \, \la' \, (\tau e^{\frac{p-1}{p}\, \tau})' \, (w_3)_x.
\end{split}
\eee
Finally,  
\begin{equation}
\label{eq-Lw}
\begin{split}
L(w) = & (w_3)_{xx} (w_1 + w_2 - 1) + (w_1)_{xx} (w_3 + w_2 - 1)
 - \frac{2p}{p-1} (w_1)_x (w_3)_x \\
&+ 2(p-1) (w_1 - 1)(w_3 -1) + 2(p-1) \, w_2 \, (w_1 + w_3 - 2) \\
&- p q \, \la \, (\tau e^{\frac{p-1}{p}\, \tau})' \, (w_1)_x + p q \, \la' \, (\tau e^{\frac{p-1}{p}\, \tau})' \,  (w_3)_x.
\end{split}
\end{equation}
Notice that $(w_1)_x=v_\la' \geq 0$ and $(w_2)_x=-v_{\la'}' \leq 0$, which implies that the last two   terms 
in \eqref{eq-Lw} are negative. We will see in what follows that those two terms play a crucial role for making
$w$ a supersolution. In fact, we  will next use \eqref{eq-Lw} to show that $w$ is a supersolution of  \eqref{eqn-u}.

\begin{lemma}
\label{lem-super}
There exists a number $\tau_0 <<0$ depending on $p$ and $\bla$ for which 
the function $w:=w_\bla^+$  defined as in \eqref{eq-def-w} is a supersolution  of equation \eqref{eqn-u},
 that is $Lw \le 0$ on $\R \times (-\infty, \tau_0]$.
\end{lemma}

\begin{proof}
The monotonicity of $v_\la, v_{\la'}$ implies that $w_1(x,\tau)$ is monotone decreasing in $x$ and $w_2(x,\tau)$ is monotone 
increasing in $x$.  Hence, the asymptotic behavior in \eqref{eqn-vla} implies that for each $\tau <<0$ there
exists a unique $x(\tau) \in \R$ at which   $w_1(x(\tau),\tau) =w_2 (x(\tau),\tau)$. In addtion,  it follows from \eqref{eqn-vla} that  
as $\tau \to -\infty$, the point $x(\tau)$ satisfies  the asymptotic behavior
\begin{equation}
\label{eq-intersection}
x(\tau) =  \frac{ \gamma_{\lambda} - \gamma_{\lambda'}}{p} \, \tau 
+ \frac{1}{\gamma_{\lambda} + \gamma_{\lambda'}}\, \left(  \ln\frac{C_{\lambda}}{C_{\lambda'}} 
+ h'\gamma_{\lambda'} - h\gamma_{\lambda} \right )  + o(1).
\end{equation}
and that  at $x=x(\tau)$ we have 
\be
\label{eqn-u12} 
w_1(x(\tau), \tau) = w_2(x(\tau), \tau) = 1- C_\bla  \, e^{d\, \tau} + o(e^{d\tau})
\ee
with
\be
\label{eq-d}
d:= \frac {\gamma_{\lambda} \gamma_{\lambda'} + (p-1)}p.
\ee
and $C_\bla$ depending on $\bla$. 

\smallskip
Trying to show that $Lw \le 0$,  for all $x  \in \R $, $\tau \leq \tau_0 <<0$, we observe first that by symmetry it is sufficient to
consider the region $x \leq x(\tau)$, $\tau \leq \tau_0$. Recalling that $z:=x -\la\tau(1 - q e^{\frac{p-1}{p}\, \tau}) + h$, 
we will next  distinguish between the three different cases  $z \le -M$, $-M \le z \le M$ and $M \le z \le x(\tau) - \la\tau(1 - q e^{\frac{p-1}{p}\, \tau}) + h$.

\begin{case}
$z \le -M$.
\end{case}

This implies $x\le -M + \la\tau (1 - q e^{\frac{p-1}{p}\, \tau}) - h$ and hence $\bar{z} \ge M - (\la + \la')\tau\, (1 - qe^{\frac{p-1}{p}\, \tau}) + h + h'$. 
By \eqref{eqn-ula1} we have that in the considered region the following asymptotics (up to a constant, whether positive or negative) hold
\[w_1(x,\tau) \sim e^{-(p-1) z}, \qquad w_2(\tau) = k e^{\frac{p-1}{p}\, \tau} + o(e^{\frac{p-1}{p}\, \tau}), \qquad w_3(x,\tau) \sim 1 + 
C_p e^{-\gamma_{\la'} \bar{z}}.\]
As in \cite{DDKS} we can argue that up to a time independent constant (whether positive or negative), in the considered region we have
\begin{equation}
\label{eq-beh-w}
w_1 \sim (w_1)_x, \qquad (w_1)_{xx} \sim (w_1)_x, \qquad (w_3)_x \sim e^{-\gamma_{\la'}\bar{z}}, \qquad (w_3)_{xx} \sim e^{-\gamma_{\la'}\bar{z}}.
\end{equation}
On the other hand, since $\bar{z} \ge M - (\la + \la')\tau\, (1 - qe^{\frac{p-1}{p}\, \tau}) + h + h'$, we have 
\begin{equation}
\label{eq-gamma-la}
e^{-\gamma_{\la'}\bar{z}} \le C\, e^{(\la + \la')\gamma_{\la'}\, \tau} = o(e^{\la'\gamma_{\la'}\tau}) = o(e^{\frac{p-1}{p}\,\tau})
\end{equation}
where the last equality holds because  $\la'\gamma_{\la'} > \frac{p-1}{p}$. 

Lets now analyze terms in \eqref{eq-Lw}. We claim the prevailing term is $-(p - 1)q\la \tau e^{\frac{p-1}{p}\, \tau} (w_1)_x < 0$, and that the others can be absorbed in this one, at least when $\tau <<0$.  It turns out the constants (whether positive or negative)  for the other terms in \eqref{eq-Lw} are not important and hence we will only emphasize the behavior. Using \eqref{eq-beh-w} and \eqref{eq-gamma-la} we get
\[(w_3)_{xx}\, (w_1 + w_2 - 1) \sim (w_3)_{xx} w_1 \sim o(e^{\frac{p-1}{p}\, \tau}) (w_1)_x \]
and 
\[(w_1)_{xx}\, (w_3 + w_2 - 1) \sim (w_1)_x  (w_3 - 1) + (w_1)_x w_2 = O(e^{\frac{p-1}{p}\, \tau})\, (w_1)_x.\]
Similarly, we get
\[- \frac{2p}{p-1} (w_1)_x (w_3)_x + 2(p-1) (w_1 - 1)(w_3 -1) + 2(p-1) w_2 (w_1 + w_3 - 2)
= O(e^{\frac{p-1}{p}\, \tau}) \, (w_1)_x.\]
Combining the above finally yields 
\[L(w) \le -(p-1)\, q  \la \, \tau e^{\frac{p-1}{p}\, \tau}  \, (w_1)_x + o(\tau e^{\frac{p-1}{p}\,\tau}  \,  (w_1)_x) <  0\]
for $\tau <<0$.

\begin{case}
$-M \le z \le M$.
\end{case}

This is equivalent to  $-M + \la\tau (1 - q e^{\frac{p-1}{p}\,\tau}) - h \le x \le M + \la\tau (1 - q e^{\frac{p-1}{p}\,\tau}) - h$ which implies that $\bar{z} \ge -M - (\la + \la')\, \tau (1 - q e^{\frac{p-1}{p}\,\tau}) + h + h'$.  By \eqref{eqn-ula1}, in the considered region we have
\[a \le w_1(x,\tau) \le b, \quad w_2(\tau) = k e^{\frac{p-1}{p}\, \tau} + o(e^{\frac{p-1}{p}\tau}), \quad w_3(x,\tau) \sim 1 + C_p e^{-\gamma_{\la'}\bar{z}},
\quad e^{-\gamma_{\la'}\bar{z}} = o(e^{\frac{p-1}{p}\, \tau})\]
for some constants $a, b > 0$.  As before, we have 
\[ (w_3)_x \sim e^{-\gamma_{\la'}\bar{z}}, \qquad (w_3)_{xx} \sim e^{-\gamma_{\la'}\bar{z}}\]
up to a constant (positive or negative). Analyzing terms in \eqref{eq-Lw} and using \eqref{eq-gamma-la}
\[a_1 \le (w_1)_x \le b_1, \qquad |(w_1)_{xx}| \le C, \qquad (w_3)_x = o(e^{\frac{p-1}{p}\, \tau}), \qquad (w_3)_{xx} = o(e^{\frac{p-1}{p}\, \tau})\]
we find 
\[(w_3)_{xx} (w_1 + w_2 - 1) = o(e^{\frac{p-1}{p}\, \tau}), \quad (w_1)_{xx} (w_3 + w_2 - 1) = O(e^{\frac{p-1}{p}\, \tau}).\]
Similarly, we get
\[- \frac{2p}{p-1} (w_1)_x (w_3)_x + 2(p-1) (w_1 - 1)(w_3 -1) + 2(p-1) \, w_2 \, (w_1 + w_3 - 2)
= O(e^{\frac{p-1}{p}\, \tau}).\]
Hence, for $-M \le z \le M$,  we conclude 
\[L(w) \leq -(p-1)\, q  \la \, \tau \, e^{\frac{p-1}{p}\, \tau}  \, (w_1)_x + o(\tau e^{\frac{p-1}{p}\,\tau}) 
\le - (p-1) q \la \, a_1 \, \tau \, e^{\frac{p-1}{p}\, \tau}   + o (\tau \, e^{\frac{p-1}{p}\,\tau}) < 0\]
for $\tau <<0$. 

\begin{case}
$M \le z \le x(\tau) - \la\tau (1 - q e^{\frac{p-1}{p}\,\tau}) + h$. 
\end{case}

Using \eqref{eq-intersection}, in this case  we have
$$M \leq z \leq x(\tau) - \la\tau (1 - q e^{\frac{p-1}{p}\,\tau}) + h=  -\frac{(\gamma_\la \gamma_{\la'} + (p-1))}{p \gamma_\la}\, \tau + C(\la,\la',h,h',p)+o(1).$$
Also, $M + \la\tau (1 - q e^{\frac{p-1}{p}\, \tau}) - h \le x \le x(\tau)$ implies that 
\[\bar{z} \ge -x(\tau) - \la'\tau (1 - q e^{\frac{p-1}{p}\, \tau}) + h' = -\frac{\gamma_\la (\la + \la')}{\gamma_\la + \gamma_{\la'}}\, \tau + C(\la,\la',h,h',p)+o(1).\]
It follows that  in the considered region both $z >>1$ and $\bar z >>1$, hence   
\[w_1 = 1 + C_p e^{-\gamma_\la z}, \qquad (w_1)_x \sim - C_p e^{-\gamma_\la z}, \qquad (w_1)_{xx} \sim C_p e^{-\gamma_\la z}\]
and 
\[w_3 = 1 + C_p' e^{-\gamma_{\la'} \bar{z}}, \qquad (w_3)_x \sim C_p' e^{-\gamma_{\la'} \bar z}, \qquad (w_3)_{xx} \sim C_p' e^{-\gamma_{\la'} \bar z}\]
up to a positive constant, with $C_p>0, C'_p >0$. 
Analyzing terms in \eqref{eq-Lw},   we have
\[
\begin{split}
(w_3)_{xx} (w_1 + w_2 - 1) &\sim C_1 \, e^{-\gamma_{\la'}\bar{z}} e^{\frac{p-1}{p}\, \tau} + C_2 \, e^{-\gamma_{\la'}\bar{z}} e^{-\gamma_\la z}  \\
&\sim (w_1)_x\, \left( C_1 \, e^{\frac{p-1}{p}\, \tau}\, e^{-\gamma_{\la'}\bar{z} + \gamma_\la z} + C_2 \,  e^{-\gamma_{\la'}\bar{z}}\right)
\end{split}
\]
with $C_1, C_2 >0$. 
Note that
\[e^{-\gamma_{\la'}\bar{z}} \le e^{\frac{\gamma_\la\gamma_{\la'} \, (\la + \la')\, \tau}{\gamma_\la + \gamma_{\la'}}} \le C e^{\frac{p-1}{p}\, \tau}\]
since $\la \gamma_{\la} \ge \frac{p-1}{p}$. Furthermore,
\[e^{-\gamma_{\la'}\bar{z} + \gamma_\la z} \le e^{\gamma_\la(x(\tau) - \la\tau) - \gamma_{\la'} (-x(\tau) - \la'\tau)} \le C\]
due to a definition of $x(\tau)$. Overall we have,
\[(w_3)_{xx}\, (w_1 + w_2 - 1) \le C \, e^{\frac{p-1}{p}\, \tau}\, (w_1)_x = o \big ( (\tau e^{\frac{p-1}{p}\, \tau})\, (w_1)_x \big ).\]
Similarly we analyze other terms in \eqref{eq-Lw} of the same form in the considered region and we also use that 
$p q \, \la' \, (\tau e^{\frac{p-1}{p}\, \tau})' \,  (w_3)_x \leq 0$,  to conclude that the dominating term is $-p q \, \la \, (\tau e^{\frac{p-1}{p}\, \tau})' \,  (w_1)_x < 0$, hence
\[L(w) \leq -(p-1)\, q  \la \, \tau e^{\frac{p-1}{p}\, \tau}  \, (w_1)_x + o(\tau e^{\frac{p-1}{p}\,\tau})   < 0\]
for $\tau <<0$. 

Finally, above analysis in cases 1-3  yields  that $L(w) \leq  0$, meaning that $w$ is a supersolution as stated in the Lemma.
\end{proof}

\section{The existence of a five parameter family of  ancient solutions}
\label{sec-solutions}

In this section we show the the existence of a five parameter ancient solution $u_\bla$ of \eqref{eqn-u}, as stated in the next theorem. 

\begin{thm}\label{thm-ubla}
For any $(\bla) \in \mathbb{R}^5$ such that $\la, \la' > 1$, $k \ge 0$, there exists an ancient solution $u_\bla$ to 
\eqref{eqn-u} defined on $\mathbb{R} \times (-\infty, T)$, for some $T = T_\bla \in (-\infty, \infty]$. Moreover,
\[w_\bla^-(x,\tau) \le u_\bla(x,\tau) \le w_\bla^+(x,\tau)\]
where $w_\bla^-$ and $w_\bla^+$ are defined by \eqref{eqn-vbla} and \eqref{eqn-wbla}, respectively.
\end{thm}

\begin{proof}
We have seen in Lemma \ref{lem-super}  that $w^+_\bla$ is a supersolution of \eqref{eqn-u} for $\tau \leq \tau_0$ with $\tau_0 <<0$.
For any $m \in \mathbb{N}$ with $-m < \tau_0$,   let  $u_m$ denote the solution of the initial value problem 
\be
\label{eqn-um}
\begin{cases}
p\, u_\tau = u\, u_{xx} - \frac p{p-1} \,  u_x^2 + (p-1) \, \big ( u^2 - u \big ) \qquad & x\in  \R, \,\,  \tau > -m \\
\, u(\cdot,-m) = w^+_\bla(\cdot, -m) \qquad & x \in  \R
\end{cases}
\ee
with exponent
$p= \frac{n+2}{n-2} >1.$  By Lemma \ref{lem-super} and the comparison principle we immediately get that 
$$u_m(x,\tau) \le w_\bla^+(x,\tau), \qquad  -m \leq \tau \leq \tau_0.$$ 
On the other hand, since $v_\la(x - \la\tau + h)$, $v_{\la'}(-x-\la'\tau+h')$ and $\xi_k(\tau)$ are all solutions to \eqref{eqn-u}, 
it follows that   $w_\bla^-(x,\tau):= \min (v_\la(x - \la\tau + h), \xi_k(\tau), v_{\la'}(-x-\la'\tau+h')) $  is a subsolution of \eqref{eqn-u}. Furthermore, we claim that 
\be\label{eqn-123}
w_\bla^-(x,\tau) \le u_m(x,\tau), \qquad \mbox{for} \,\, \tau < 0
\ee
which readily gives 
$$w_\bla^-(x,-m) \le u_m(x,-m).$$  
To see  \eqref{eqn-123}, we  recall that  $v_\la$ is monotone decreasing. Also that since all $v_\la,  v_{\la'}$\and $\xi_k$
are larger than equal to one, we have
$$w^+_\bla \geq \max \big ( v_\la( x- \la  \tau \, (1 - q \, e^{ \frac {p-1}p \tau} )  +h), \,  \xi_k(\tau), \,  
v_\la( - x- \la  \tau \, (1 - q  e^{ \frac {p-1}p\tau} )  +h) \big ).$$
Hence, on the regions where 
$w_\bla^-(x,\tau) = v_\la(x - \la \tau + h)$  and $w_\bla^-(x,\tau) = v_{\la'}(-x - \la' \tau + h')$,  respectively,  we have 
\[w_\bla^-(x,\tau) = v_\la(x -\la \tau  + h) \le v_\la(x -  \la \tau \, ( 1  -  q e^{\frac{p-1}{p}\,  \tau} ) ) \le w_\bla^+(x,\tau)\]
and
\[w_\bla^-(x,\tau) = v_{\la'}(-x - \la' \tau + h') \le v_{\la'}(-x - \la' \tau\,  (1-  q e^{\frac{p-1}{p} \tau}) + h') \le w_\bla^+(x,\tau).\]
Finally, on the region where  $w_\bla^-(x,\tau) = \xi_k(\tau)$ we  immediately have
$$w_\bla^-(x,\tau) = \xi_k(\tau)  \le w_\bla^+(x,\tau).$$
Hence,
$w_\bla^-(x,-m) \le \xi_k(-m) \le w_\bla^+(x,-m) = u_m(x,-m)$ 
and by the   comparison principle we  $w_\bla^-(x,\tau) \le u_m(x,\tau)$,  for all $\tau \ge -m$.  We conclude from the discussion above that 
$$w_\bla^-(x,\tau) \le u_m(x,\tau) \leq w_\bla^+(x,\tau), \qquad  -m \leq \tau \leq \tau_0.$$
It follows from standard arguments in quasilinear parabolic equations that the sequence of solutions $\{ u_m \}_{\{ m < \tau_0 \}}$
is equicontinuous on compact subsets of $\R \times (-\infty, \tau_0)$, hence passing to a subsequence $u_{m_k}$ it converges 
uniformly on compact subsets of $\R \times (-\infty, \tau_0)$ to a smooth ancient solution $u_\bla$ of \eqref{eqn-u}. In addition, 
$u$ satisfies  
\be\label{eqn-inequ}
w_\bla^-(x,\tau) \le u_\bla(x,\tau) \leq w_\bla^+(x,\tau), \qquad  -m \leq \tau \leq \tau_0.
\ee  
\end{proof}

We will next observe that $u_\bla$ defines a smooth ancient  solution of the rescaled Yamabe flow on $S^n \times (-\infty, T_\bla)$,
for some {\em maximal time}  $T_\bla \in (-\infty, +\infty]$ at which $u_\bla < +\infty.$

\begin{cor}\label{cor-phibla}
Set  $\vp_\bla:= u_\bla^{-\frac 1{p-1}}$, where $u_\bla$ is an  ancient  solution of \eqref{eqn-u} given by  Theorem \ref{thm-ubla}.  
Then, $\vp_\bla$  defines a solution of \eqref{eqn-vp}  which is equivalent to having a smooth solution of the rescaled Yamabe flow
 $g_t = - (R-1)\, g$ on $S^n \times (-\infty, T_\bla)$.
\end{cor}

\begin{proof} It is clear that $\vp_\bla:= u_\bla^{-\frac 1{p-1}}$ is a smooth ancient solution of \eqref{eqn-vp} which in addition satisfies
\be\label{eqn-inequ2}
{w_\bla^+}(x,\tau)^{-\frac 1{p-1}} \le \vp_\bla(x,\tau) \leq {w_\bla^-}(x,\tau)^{-\frac 1{p-1}}, \qquad  x \in \R, \,\, -m \leq \tau \leq \tau_0.
\ee  
Notice next that  ${w_\bla^-}(\cdot,\tau)^{-\frac 1{p-1}}$ and ${w_\bla^+}(\cdot,\tau)^{-\frac 1{p-1}}$ define 
positive smooth metrics, when  lifted on $S^n$. Thus,   \eqref{eqn-inequ2} implies that $\vp_\bla(x,\tau)$,  
when lifted on $S^n$,  also defines a smooth metric which equivalent (up to dilation) to a solution of the rescaled Yamabe flow
$g_t = - (R-1)\, g$. Now standard results imply  that there exists a maximal time $T_\bla \in (-\infty, +\infty]$ up to which this solution is defined,
which means that $\vp_\bla (x, t) >0$ for all $t < T_\bla$.  

\end{proof} 

\section{Geometric properties of solutions}
\label{sec-geom}

In this last section we will derive the geometric properties of the ancient solution of the equation \eqref{eqn-vp}  as constructed 
in Theorem \ref{thm-ubla}. We have observed in Corollary \ref{cor-phibla}
that the  one parameter family of metrics    $g_\bla(\tau):= \vp_\bla^\mmm(\cdot, \tau) \, g_{_{cyl}}$   can be lifted to a smooth one parameter family of metrics on $S^n \times (-\infty, T_\bla)$ which defines (up to dilation by constants) an ancient  rotationally symmetric  solution  of the   {\em rescaled Yamabe flow}  on $S^n$, equation 
\begin{equation}
\label{eq-RYF}
\frac{\partial}{\partial \tau} g = -(R - 1) \, g.
\end{equation}

\smallskip
We next prove the following result concerning the behavior of the Riemannian curvature  of the metric $g_\bla(\tau)$ near $\tau=  -\infty$. 

\begin{thm}\label{thm2}
The solution  $g_\bla(\tau):=\vp_\bla^\mmm(\cdot, \tau) \, g_{_{cyl}}$ defines a type I ancient solution to the Yamabe flow in the sense that, for any $\tau_0 < T_\bla$,  the norm of its curvature operator satisfies the uniform bound 
\be \label{eqn-btype1} \| { \mbox{\em Rm}} \, (g_\bla)\| \le C, \qquad \mbox{for all} \,\,  \tau \leq \tau_0. 
\ee
\end{thm}

\begin{remark}{ The statement of Theorem \ref{thm2}  exactly means that the unrescaled flow \eqref{eq-YF}, whose scaling by $|t|$ yields to the equation 
\eqref{eq-RYF}, is a type I ancient solution according to the Definition \ref{defn-ancient}}. 

\end{remark}

\begin{proof}[Proof  of Theorem \ref{thm2}]
It is sufficient to prove the bound \eqref{eqn-btype1} for $\tau_0 <<0$. Since our metric $g_\bla$ is conformally flat, the norm $\|\mbox{\text  \rem}\|$ of its curvature operator   can be expressed in terms of  powers (positive or negative) of the conformal factor and  its first and second order derivatives. On the other hand,  since the conformal 
factor  (in any parametrization) satisfies a quasilinear parabolic equation,   it follows by  standard parabolic estimates,   that 
  uniform upper and lower bounds away from zero on the conformal factor imply  uniform bounds on all its derivatives and therefore the
 desired  uniform bound on $\|  { \mbox{\text Rm}}\|$.  By Corollary  \ref{cor-phibla}, we have  ${\displaystyle g_\bla=\vp^{\frac 4{n-2}}\, 
 g_{cyl}}$, where $\vp$ is a solution of \eqref{eqn-vp} and satisfies $0 < \vp < 1$,  for all $\tau < T_\bla$. 
Notice that in order to simplify the notation we have dropped the index $\bla$ from $\vp_\bla$ and simply denote it by $\vp$.  
 It follows by the discussion
 above that in the  region where $1/2 < \vp  <1$ the bound  $\| { \mbox{\text  Rm}}\|  \leq C$ readily holds. However,  since
 $\vp(x,\tau) \to 0$, as $x \to \pm \infty$, in the region where $0 < \vp<1/2$ we will obtain the desired bound  by lifting 
 the metric $g_\bla$ on $\R^n$ and showing that in the considered region if $g_\bla= \hat  \vp^{\frac 4{n-2}}\, g_{_{\R^n}}$, then
$c \leq \hat \vp \leq C$ for some positive constants $c, C$ which are uniform in $\tau$.  
 
\smallskip
 
Observe  that  is sufficient  to establish  the  uniform bound on $\|{ \mbox{\text  Rm}}\|$  
 for $x \le x(\tau)$, where $x(\tau)$ is given by \eqref{eq-intersection}, since in the remaining region the estimate could be proved in a similar manner. We go from cylindrical to polar coordinates via the following coordinate change,
\begin{equation}
\label{eq-coord-change}
\vP(z,\tau) = \hat{\vp}(y,\tau)\, |y|^{\frac{2}{p-1}}, \qquad r = |y| = e^{\frac{p-1}{2}\, z}
\end{equation}
where $z := x - \la\tau + h$ and  $\vP(z,\tau) := \vp(z +\la\tau - h)$. It follows by a direct calculation that  $\hat{\vp}(y,\tau)$ satisfies the equation 
\begin{equation}
\label{eq-u-polar}
(\hat{\vp}^p)_{\tau} = \alpha \, \Delta \hat{\vp} + \beta\,  r \, \hat{\vp}_r + \gamma \, \hat{\vp}
\end{equation}
for some  constants $\alpha>0$ and $\beta, \gamma$. For any $M >0$, the compact region $|y| \le 2M$ corresponds in cylindrical coordinates to the  region 
$$x \le \la\tau - h + \frac{2}{p-1}\, \log 2M =: \bar{x}_M(\tau).$$  Note that
$\la p > \gamma_\la  > \gamma_\la - \gamma_{\la'}$, implying that $\bar x_M(\tau) << x(\tau)$ holds, for $\tau <<0$.

Estimate \eqref{eqn-inequ2} is crucial in proving our bound. 
Let us  look first at the region $x \le \bar{x}_M(\tau)$, which in polar coordinates corresponds to the region $|y| \le M$. 
It follows from \eqref{eqn-inequ2} that in this region we have
\[(1 - o_M(1))\, \psi_\la \le \vp_\bla \le \psi_\la\]
for $\tau \leq \tau_0$ with $\tau_0 <<0$, where we recall that  $\psi_\lambda$ is the traveling wave solution of \eqref{eqn-vp2} which satisfies \eqref{eqn-vpla}. 
In polar coordinates the above abound corresponds to 
\begin{equation}
\label{eq-par-polar}
c_M \le \hat{\vp} \le C, \qquad \mbox{for} \,\,\, |y| \le M, \quad \tau \leq \tau_0.
\end{equation}
Having \eqref{eq-par-polar}, equation \eqref{eq-u-polar} is parabolic equation for $(y,\tau) \in B(0,2M) \times (-\infty, \tau_0)$, so standard parabolic estimates applied to equation \eqref{eq-u-polar} imply that we have uniform bounds on all the derivatives of $\hat{\vp}$ in the region $B(0,\frac{3M}{2}) \times (-\infty, 2\tau_0)$. Since $\hat{\vp}^{\frac{4}{n-2}}$ is the conformal factor of our metric $g_\bla$ in polar coordinates, by the previous discussion we have
\[\| { \mbox{\text  Rm}}(y,\tau)\| \le C, \qquad |y| \leq M  \qquad \tau\le 2\tau_0\]
for a uniform constant $C$. Equivalently we have a uniform curvature bound, in cylindrical coordinates, for all $x \le \bar{x}_M(\tau)$. 
Observe that this estimate implies that the is curvature uniformly bounded in the tip region of our ancient solution.

Now fix $M >0$ and let us  focus on the region $\bar{x}_M(\tau) \le x \le x(\tau)$ where 
 our solution that turns out to have the asymptotics of a cylindrical metric at $\tau \to -\infty$.  More precisely,  we have ${\displaystyle g_\bla=\vp^{\frac 4{n-2}}\, 
 g_{cyl}}$, where $\vp$ is a solution of \eqref{eqn-vp} and   by \eqref{eqn-inequ2} this region we may choose $\tau_0 <<0$ such that  
\[1/2 \le \vp(x,\tau) \le 1, \qquad \mbox{for} \,\, \tau \leq \tau_0.\]
The equation satisfied by $\vp$ is therefore uniformly parabolic and hence we have uniform estimates on the derivatives of $\vp$ in the inner region. As a result we have a uniform bound on $\| { \mbox{\text  Rm}}\|$ for all $\tau \le \tau_0$ in that region as well. This concludes the proof of our bound. 
\end{proof}

\centerline{\bf Acknowledgements}

\noindent P. Daskalopoulos  has been partially supported by NSF grant DMS-1266172.\\
M. del Pino has been
supported by grants Fondecyt  1150066  Fondo Basal CMM \\and Millenium Nucleus CAPDE NC130017.\\
N. Sesum has been partially supported by NSF grant  DMS-1056387.

\end{document}